\documentclass[reqno,12pt]{amsart}
\usepackage[T1]{fontenc}
\usepackage[utf8]{inputenc}
\usepackage[english]{babel}
\usepackage{amssymb,amsmath,amsthm,amscd,latexsym,amsfonts,xcolor,enumerate,hyperref,comment,longtable,cleveref}

\usepackage{times}
\usepackage{cite}
\usepackage{pdflscape}
\usepackage[mathcal]{euscript}
\usepackage{tikz}
\usepackage{hyperref}
\usepackage{cancel}
\usepackage{stmaryrd}
\usepackage{dsfont}

\numberwithin{equation}{section}

\usetikzlibrary{arrows}

\DeclareMathOperator{\gr}{gr}
\DeclareMathOperator{\rank}{rank}
 
  \DeclareMathOperator{\lin}{lin}

\usepackage[a4paper,top=3cm, bottom=3cm, left=3cm, right=3cm]{geometry}

\newtheorem{pr}[subsection]{Proposition}

\newtheorem{teo}[subsection]{Theorem}

\theoremstyle{definition}
\newtheorem{de}[subsection]{Definition}

\theoremstyle{remark}

\newtheorem{cons}{Consequence}

\usepackage{stmaryrd}

\begin{document}

\title[Some classes of nilpotent associative algebras]
{Some classes of nilpotent associative algebras}

\author{I.A.~Karimjanov, M.~Ladra}
\address{[I.A.~Karimjanov, M.~Ladra] Department of Mathematics, Institute of Mathematics,  University of Santiago de Compostela, 15782, Spain}
\email{iqboli@gmail.com, manuel.ladra@usc.es}

\thanks{This work was supported by Agencia Estatal de Investigaci\'on (Spain), grant MTM2016-79661-P (European FEDER support included, UE)}

\begin{abstract}
In this paper we classify filiform associative algebras of degree $k$ over a field of characteristic zero. Moreover, we also classify naturally graded complex filiform and quasi-filiform nilpotent associative algebras which are described by the characteristic sequence $C(\mathcal{A})=(n-2,1,1)$ or $C(\mathcal{A})=(n-2,2)$.
\end{abstract}


\maketitle

\section{Introduction}

The classification of associative algebras is an old and often recurring problem.  Associative algebras are one of the classical algebras that have extensively been studied and found to be related to other classical algebras like Lie and Jordan algebras.

In the 1890s Cartan, Frobenius and Molien proved (independently) the following fundamental structure theorem for a finite-dimensional associative algebra $\mathcal{A}$ over the real or complex numbers:

\begin{itemize}
  \item $\mathcal{A}=S\oplus N$, where $N$ is nilpotent and $S$ is semi-simple;
  \item $S=C_1\oplus C_2\oplus\dots\oplus C_n$, where $C_i$ are simple algebras;
  \item $C_i=M_{n_i}(D_i)$, the algebra of $n_i\times n_i$ matrices with entries from a division algebra $D_i$.
\end{itemize}

The first researches into the classification of low-dimensional associative algebras was done by Peirce \cite{Pr}. Hazlett \cite{Haz} classified nilpotent algebras of dimension less and equal then $4$ over the complex numbers. Later, Kruse and Price \cite{Kr} classified nilpotent associative algebras of dimension less and equal then $4$ over any field. Mazzola \cite{Maz1} published
his results on associative unitary algebras of dimension $5$ over algebraic closed fields of characteristic different from $2$, and on nilpotent commutative associative algebras of dimension less and equal $5$ over algebraic closed fields of characteristic different from $2,3$ \cite{Maz2}. Poonen \cite{Poo} analyzed nilpotent commutative associative algebras of dimension less and equal $5$ over  any algebraic closed field
 and Eick and Moede \cite{Eick, Eick2} initiated on an arbitrary field a coclass
theory for nilpotent associative algebras, giving a different view on their classification since  they use the coclass as primary invariant, describing  algorithms that allow to investigate the graph  associated with the nilpotent
associative algebras of a certain coclass  for a finite field. Recently, De Graaf classified nilpotent associative algebras of dimension less and equal than $4$ by using method of central extensions \cite{Graaf}.

Most classification problems of finite-dimensional associative algebras have been
studied for certain properties of associative algebras while the complete classification of associative algebras in general is still an open problem.  The theory of finite-dimensional
associative algebras is the one of the ancient areas of the
modern algebra. It originates primarily from works Hamilton,
who discovered the famous quaternions and Cayley, who
developed the theory of matrices. Later, the structural theory
of finite-dimensional associative algebras have been treated by
several mathematicians, notably B. Pierce, C. S. Pierce,
Clifford, Weierstrass, Dedekind, Jordan, Frobenius. At the end
of 19th century, T. Molien and E. Cartan described semi-simple
algebras over the fields of the complex and real numbers.

The purpose of this paper is studying naturally graded arbitrary dimensional associative algebras of nilpotency class $n$ and $n-1$. Similar results for Lie and Leibniz algebras were obtained in the works \cite{An,Gar,Ver,Gom,Cam1,Cam2,Ay}.

This paper is organized as follows. In Section~\ref{S:prel} we provide some basic concepts needed
for this study. In Section~\ref{S:fil} we first classify naturally graded filiform associative algebras  of degree $p$, then describe filiform associative algebras of degree $p$ over a field of characteristic zero. Moreover, we give the classification of $n$-dimensional complex filiform associative algebras. The last section is devoted classifying naturally graded complex quasi-filiform nilpotent associative algebras which are described by the characteristic sequence $C(\mathcal{A})=(n-2,1,1)$ or $C(\mathcal{A})=(n-2,2)$.

\section{Preliminaries}\label{S:prel}
For an algebra $\mathcal{A}$ of an arbitrary variety, we consider the series
\[
\mathcal{A}^1=\mathcal{A}, \qquad \ \mathcal{A}^{i+1}=\sum\limits_{k=1}^{i}\mathcal{A}^k \mathcal{A}^{i+1-k}, \qquad i\geq 1.
\]

We say that  an  algebra $\mathcal{A}$ is \emph{nilpotent} if $\mathcal{A}^{i}=0$ for some $i \in \mathbb{N}$. The smallest integer satisfying $\mathcal{A}^{i}=0$ is called the  \emph{index of nilpotency} or \emph{nilindex} of $\mathcal{A}$.

\begin{de}
An $n$-dimensional algebra $\mathcal{A}$ is called null-filiform if $\dim \mathcal{A}^i=(n+ 1)-i,\ 1\leq i\leq n+1$.
\end{de}

It is easy to see that an algebra has a maximum nilpotency index if and only if it is null-filiform. For a nilpotent algebra, the condition of null-filiformity is equivalent to the condition that the algebra is one-generated.

All null-filiform associative algebras were described in  \cite[Proposition 5.3]{MO}.

\begin{teo}[\cite{MO}] An arbitrary $n$-dimensional null-filiform associative algebra is isomorphic to the algebra:
\[\mu_0^n : \quad e_i e_j= e_{i+j}, \quad 2\leq i+j\leq n,\]
where $\{e_1, e_2, \dots, e_n\}$ is a basis of the algebra $\mathcal{A}$ and the omitted products vanish.
\end{teo}

\begin{de}
An $n$-dimensional algebra is called filiform if $\dim(\mathcal{A}^i)=n-i, \ 2\leq i \leq n$.
\end{de}

\begin{de} An $n$-dimensional associative algebra $\mathcal{A}$ is called quasi-filiform algebra if $\mathcal{A}^{n-2}\neq0$ and $\mathcal{A}^{n-1}=0$.
\end{de}

\begin{de} Given a nilpotent associative algebra $\mathcal{A}$, put
$\mathcal{A}_i=\mathcal{A}^i/\mathcal{A}^{i+1}, \ 1 \leq i\leq k-1$, and $\gr \mathcal{A} = \mathcal{A}_1 \oplus
\mathcal{A}_2\oplus\dots \oplus \mathcal{A}_{k}$. Then $[\mathcal{A}_i,\mathcal{A}_j]\subseteq \mathcal{A}_{i+j}$ and we
obtain the graded algebra $\gr \mathcal{A}$. If $\gr \mathcal{A}$ and $\mathcal{A}$ are isomorphic,
denoted by $\gr \mathcal{A}\cong \mathcal{A}$, we say that the algebra $\mathcal{A}$ is naturally
graded.
\end{de}

For any element $x$ of $\mathcal{A}$ we define the  left multiplication operator as
\[L_x \colon  \mathcal{A} \rightarrow \mathcal{A},  \quad z \mapsto xz, \ z\in\mathcal{A}.\]

Let us $x\in\mathcal{A}\setminus\mathcal{A}^2$ and for the nilpotent left multiplication operator $L_x$, define the decreasing sequence $C(x)=(n_1,n_2, \dots, n_k)$ that consists of the dimensions of the Jordan blocks of the operator $L_x$. Endow the set of these sequences with the lexicographic order, i.e. $C(x)=(n_1,n_2, \dots, n_k)\leq C(y)=(m_1,m_2, \dots, m_s)$ means that there is an $i\in\mathbb{N}$ such that $n_j=m_j$ for all $j<i$ and $n_i<m_i$.

\begin{de} The sequence $C(\mathcal{A})=\max_{x\in\mathcal{A}\setminus\mathcal{A}^2}C(x)$ is defined to be the characteristic sequence of the algebra $\mathcal{A}$.
\end{de}

\begin{de}
The left center of $\mathcal{A}$ is denoted by $Z^l(\mathcal{A})=\{x\in\mathcal{A} \mid x a=0 \ \text{for all} \ a\in\mathcal{A}\}$ and the right center of $\mathcal{A}$ is denoted by $Z^r(\mathcal{A})=\{x\in\mathcal{A} \mid a x=0 \ \text{for all} \ a\in\mathcal{A}\}$. The center of $\mathcal{A}$ is $Z(\mathcal{A})=Z^l(\mathcal{A})\cap Z^r(\mathcal{A})$.
\end{de}

\section{Filiform associative algebras of degree $p$} \label{S:fil}

Now we define filiform algebras of degree $p$.

\begin{de}
An $n$-dimensional associative algebra $\mathcal{A}$ is called filiform of degree $p$ if $\dim(\mathcal{A}^i)=n-p+1-i, \ 1\leq i \leq n-p+1$.
\end{de}


\begin{teo}\label{natpfil}
 Let $\mathcal{A}$ be a naturally graded filiform associative algebra of  dimension $n (n>p+2)$ of degree $p$  over a field $\mathbb{F}$ characteristic zero. Then, $\mathcal{A}$ isomorphic to $\mu_0^{n-p}\oplus\mathbb{F}^p$.

\end{teo}

\begin{proof}Let $\mathcal{A}$ be a naturally graded filiform associative algebra of dimension $n$ of degree $p$  and let $\{e_1, e_2,\dots, e_{n-p}, f_1,f_2, \dots, f_p \}$ be a basis for $\mathcal{A}$ such that $\{e_1,f_1,f_2,\dots,f_p\}\subset \mathcal{A}_1, e_2\in\mathcal{A}_2,\dots, e_{n-p}\in\mathcal{A}_{n-p}$
(such a basis can be chosen always). Moreover $\mathcal{A}_i\mathcal{A}_j\subseteq A_{i+j}$.

\textbf{Case 1.} Let $e_1e_1\neq0$. Then we can assume $e_1e_1=e_2$. We denote
\[e_1f_i=\alpha_ie_2, \quad f_ie_1=\beta_ie_2, \quad f_if_j=\gamma_{ij}e_2, \quad e_1e_2=\delta_1e_3, \quad 1\leq i,j\leq p.\]

The identity $e_1(e_1e_1)=(e_1e_1)e_1$
implies
$e_1e_2=e_2e_1=\delta_1e_3$.

Considering
\begin{align*}
e_2f_i & =(e_1e_1)f_i=e_1(e_1f_i)=\alpha_ie_1e_2=\alpha_i\delta_1e_3, \\
f_ie_2 & =f_i(e_1e_1)=(f_ie_1)e_1=\beta_ie_1e_2=\beta_i\delta_1e_3,
\end{align*}
we obtain $e_2f_i=\alpha_i\delta_1e_3$ and $f_ie_2=\beta_i\delta_1e_3$ and it  follows $\delta_1\neq0$ otherwise $e_3\notin \mathcal{A}^3$.

Consider
\[e_1(f_ie_1)=\beta_ie_1e_2=\beta_i\delta_1e_3.\]
On the other hand
\[e_1(f_ie_1)=(e_1f_i)e_1=\alpha_ie_2e_1=\alpha_i\delta_1e_3.\]
Consequently, $\beta_i=\alpha_i$, i.e. $e_1f_i=f_ie_1=\alpha_ie_2$ where $1\leq i\leq p$.

The equality $e_1(f_if_j)=(e_1f_i)f_j$  implies $\gamma_{ij}=\alpha_i\alpha_j$, and hence
\[f_if_j=\alpha_i\alpha_je_2, \quad 1\leq i,j\leq p.\]

So, we obtain that
\begin{align*}
e_1e_1 & =e_2, \qquad e_1f_i=f_ie_1=\alpha_ie_2, \qquad f_if_j=\alpha_i\alpha_je_2,\\
e_1e_2 &=e_2e_1=\delta_1e_3, \qquad e_2f_i=f_ie_2=\alpha_i\delta_1e_3,
\end{align*}
where $1\leq i,j\leq p$.

Let us take the change of basis as follows:
\[e_1^\prime=e_1, \quad e_2^\prime=e_2, \quad e_3^\prime=\delta_1e_3, \quad e_i^\prime=e_i, \quad f_j^\prime=-\alpha_je_1+f_j,\]
where $4\leq i\leq n-p$ and $1\leq j\leq p$.

Then
\begin{align*}
e_1^\prime e_1^\prime &=e_2^\prime, \\
e_1^\prime e_2^\prime &=e_2^\prime e_1^\prime=e_3^\prime,\\
f_i^\prime f_j^\prime & =(-\alpha_ie_1+f_i)(-\alpha_je_1+f_j)=\alpha_i\alpha_je_2-\alpha_i\alpha_je_2-\alpha_i\alpha_je_2+\alpha_i\alpha_je_2=0, \\
e_1^\prime f_i^\prime &=e_1(-\alpha_ie_1+f_i)=-\alpha_ie_2+\alpha_ie_2=0, \\
e_2^\prime f_i^\prime &=e_2(-\alpha_ie_1+f_i)=-\alpha_i\delta_1e_3+\alpha_i\delta_1e_3=0,
\end{align*}
where $1\leq i,j\leq p$.

 We obtain the  multiplication table:
 \[e_1e_1=e_2, \quad e_1e_2=e_2e_1=e_3, \quad e_1f_i=f_ie_1=e_2f_i=f_ie_2=f_if_j=0,\]
 where $1\leq i,j\leq p$.

Let us denote that $e_1e_3=\delta_2e_4$. Then from the equalities
\begin{align*}
e_1x_3&=e_1(e_2e_1)=(e_1e_2)e_1=e_3e_1, \qquad f_ie_3=d_i(e_1e_2)=(f_ie_1)e_2=0,\\
e_3f_i&=(e_2e_1)f_i=e_2(e_1f_i)=0,
\end{align*}
we obtain $f_ie_3=e_3f_i=0, \ 1\leq i\leq p$ and $e_1e_3=e_3e_1=\delta_2e_4$. It follows $\delta_2\neq0$ else $e_4\notin \mathcal{A}^4$. Note that taking the change $e_4^\prime=\delta_2e_4$ we can assume that
\[e_1e_3=e_3e_1=e_4, \quad f_ie_3=e_3f_i=0, \quad 1\leq i\leq p.\]
Analogously by induction we can obtain that
\[e_1e_i=e_ie_1=e_{i+1}, \quad f_je_i=e_if_j=0, \quad 4\leq i\leq n-p, \quad 1\leq j\leq p.\]
 Moreover, we have $e_{n-p}\in Z(\mathcal{A})$.

From the chain of equalities
\begin{align*}
e_ie_j&=e_i(e_1e_{j-1})=(e_ie_1)e_{j-1}=e_{i+1}e_{j-1}=\cdots\\
&=e_{i+j-2}e_2=e_{i+j-2}(e_1e_1)=(e_{i+j-2}e_1)e_1=e_{i+j},
\end{align*}
where $2\leq i+j\leq n-p$, it is easy to check $e_ie_j=0$ with $i+j> n-p$.

Hence we obtain the algebra with the multiplications table is:
\[\mu: \quad e_ie_j=e_{i+j}, \quad 2\leq i+j\leq n-p.\]

\textbf{Case 2.}  Let $e_1f_1\neq0$. Then without loss of generality, and by virtue of the symmetry  $e_1e_1=f_if_i=0$ where $1\leq i\leq p$,
 we may assume that $e_1f_1=e_2$ and denote parameter  $ f_1e_1=\gamma e_2$.
Otherwise, we will get into the Case 1. Consider the next change basis $e_1^\prime=e_1+f_1$ then $e_1^\prime e_1^\prime=(1+\gamma)e_2$,
follows $\gamma=-1$ otherwise $e_1^\prime e_1^\prime\neq0$ and it is Case 1.
 So,  we have \[e_1f_1=-f_1e_1=e_2, \quad e_1e_1=f_if_i=0, \quad 1\leq i\leq p.\] Denote by $e_1f_i=\alpha_i e_2$ and $f_ie_1=\beta_i e_2$, where $2\leq i\leq p$.

Let us consider
\begin{align*}
e_1e_2 &= e_1(e_1f_1)=(e_1e_1)f_1=0, && e_2e_1=-(f_1e_1)e_1=-f_1(e_1e_1)=0,\\
f_1e_2 &= f_1(e_1f_1)=(f_1e_1)f_1=-e_2f_1,&& e_2f_1=(e_1f_1)f_1=e_1(f_1f_1)=0,\\
f_ie_2 &= f_i(e_1f_1)=(f_ie_1)f_1=\beta_ie_2f_1=0, &&  \\
e_2f_i&=-(f_1e_1)f_i=-f_1(e_1f_i)=-\alpha_i f_1e_2=0,
\end{align*}
where $2\leq i\leq p$.

Hence we have
\[e_1e_2=e_2e_1=e_2f_i=f_ie_2=0, \quad 1\leq i\leq p,\]
it follows $x_3\notin \mathcal{A}^3$, it is contradiction. This case exists only if $\dim\mathcal{A}=p+2$ and the nonzero multiplications are
 \[e_1f_1=-f_1e_1=e_2.\]
\end{proof}

\begin{cons} In every naturally graded filiform associative algebra of dimension $p+2$ of degree $p$  over a field $\mathbb{F}$ of characteristic zero,
 there are two non-isomorphic algebras with a basis $\{e_1,e_2,f_1,f_2,\dots,f_p\}$ and the following
multiplications:
\begin{align*}
\mu_0^2\oplus\mathbb{F}^p:& \quad e_1e_1=e_2, \\
\mathcal{H}_1\oplus\mathbb{F}^{p-1}:& \quad e_1f_1=-f_1e_1=e_2,
\end{align*}
where $\mathcal{H}_1$ is the three-dimensional Heisenberg algebra.
\end{cons}

\begin{teo}\label{pfil}
Every filiform associative algebra of  dimension $n>p+2$  of degree $p$  over a field $\mathbb{F}$ of characteristic zero is isomorphic to the algebra
\[\mu^\prime:\left\{\begin{array}{ll}
  e_ie_j=e_{i+j},  & 2\leq i+j\leq n-p, \\
 e_1f_1=\alpha e_{n-p},  & \alpha\in\{0,1\}, \\
  f_if_j=\beta_{i,j}e_{n-p}, & 1\leq i,j\leq p,\\
  \end{array}\right.\]
where $\beta_{i,j}\in\mathbb{F}$, and the other products vanish.
\end{teo}
\begin{proof}
By Theorem~\ref{natpfil} every $n$-dimensional filiform associative algebras of degree $p$ over a field $\mathbb{F}$ of characteristic zero is isomorphic to the algebra $\mu^\prime\cong \mu+\mu_1$ where \[\mu_1(e_ie_j)\in \lin(e_{i+j+1},\dots,e_{n-p}),\qquad \mu_1(f_kf_s)\in \lin(e_{3},\dots,e_{n-p}),\]
\[\mu_1(e_tf_k),\mu_1(f_ke_t)\in \lin(e_{t+2},\dots,e_{n-p})\]  for $2\leq i+j\leq n-p-1, \ 1\leq k,s\leq p, \ 1\leq t\leq n-p$. It is known that $e_{n-p}\in Z(\mathcal{A})$.

We have in $\mu^\prime$
\[ e_1e_k=e_{k+1}+\sum\limits_{s=k+2}^{n-p}\beta_{k,s}e_s,  \quad 1\leq k\leq n-p-1. \]

Let us take the change of basis in the following form:
\[e_1^\prime=e_1, \quad  e_k^\prime=e^\prime_1e^\prime_{k-1}, \quad 2\leq k\leq n-p, \quad f_s^\prime=f_s, \quad 1\leq s\leq p.\]
Thus, we obtain the following  multiplication table:
\[e_1e_k=e_{k+1}, \quad 1\leq k\leq n-3.\]

From the equality $e_1e_2=e_1(e_1e_1)=(e_1e_1)e_1=e_2e_1$
we derive $e_1e_2=e_2e_1$.

Now we will prove the following equalities by an induction on $k$:
\begin{equation}\label{abc}
e_1e_k=e_ke_1, \quad 2\leq k\leq n-p-1
\end{equation}
Obviously, the equality holds for $k=2$. Let us assume that the equality \eqref{abc} holds for $2 < k < n-p-1$, and we
shall prove it for $k+1$:
\[e_1e_{k+1}=e_1(e_1e_k)=e_1(e_ke_1)=(e_1e_k)e_1=e_{k+1}e_1;\]
so the induction proves the equalities \eqref{abc} for any $k, \ 2\leq k\leq n-p-1$.

Let us introduce the notations:
\[e_1f_i=\sum\limits_{s=3}^{n-p-1}\gamma_{is}e_s+\alpha_ie_{n-p}, \quad  f_ie_1=\sum\limits_{s=3}^{n-p}\delta_{is}e_s\]
where $1\leq i\leq p$.

From the equalities
\begin{align*}
\sum\limits_{s=3}^{n-p-1}\delta_{is}e_{s+1}& =e_1(\sum\limits_{s=3}^{n-p}\delta_{is}e_s)=e_1(f_ie_1)=
(e_1f_i)e_1=(\sum\limits_{s=3}^{n-p-1}\gamma_{is}e_s+\alpha_ie_{n-p})e_1\\
&=\sum\limits_{s=3}^{n-p-1}\gamma_{is}e_{s+1},
\end{align*}
we get $\delta_{is}=\gamma_{is}$ where $3\leq s\leq n-p-1$ and $1\leq i\leq p$.

Note that taking the change of basis
\[f_i^\prime=f_i-\sum\limits_{k=3}^{n-p}\delta_{ik}e_{k-1},\]
we can assume that
\[e_1f_i=\alpha_ie_{n-p}, \quad f_ie_1=0, \quad  1\leq i\leq p.\]

Let us suppose that $\alpha_k\neq0$ for some $k$. Then by changing
\[f_i^\prime=f_i-\frac{\alpha_i}{\alpha_k}f_k, \ \ i\neq k, \qquad f_k^\prime=\frac{1}{\alpha_k}f_k,\]
we obtain $e_1f_k=e_{n-p}$ and $e_1f_i=0$ for $i\neq k$. Hence,  without loss of generality we can assume
\[e_1f_1=e_{n-p}, \qquad e_1f_i=0, \quad 2\leq i\leq p.\]

Let us denote
\[f_if_j=\sum\limits_{k=3}^{n-p-1}\eta_{i,j,k}e_k+\beta_{i,j}e_{n-p}, \quad 1\leq i,j\leq p.\]

The equalities
\[0=f_i(f_je_1)=(f_if_j)e_1=\sum\limits_{k=3}^{n-p}\eta_{i,j,k}e_ke_1=\sum\limits_{k=4}^{n-p}\eta_{i,j,k-1}e_{k}\]
yield  $\eta_{i,j,k}=0$, and it follows
\[f_if_j=\beta_{i,j}e_{n-p}, \quad 1\leq i,j \leq p.\]

Considering the identities
\begin{align*}
e_if_j&=(e_{i-1}e_1)f_j=e_{i-1}(e_1f_j)=\alpha_ie_{i-1}e_{n-p}=0, \\
f_je_i& =f_j(e_1e_{i-1})=(f_je_1)e_{i-1}=0,
\end{align*}
we can deduce that \[e_if_j=f_je_i=0, \quad 2\leq i\leq n-p, \quad 1\leq j\leq p.\]

Finally from the following chain of equalities
\begin{align*}
e_ie_j&=e_i(e_1e_{j-1})=(e_ie_1)e_{j-1}=e_{i+1}e_{j-1}=\dots=e_{i+j-2}e_2=e_{i+j-2}(e_1e_1)\\
&=(e_{i+j-2}e_1)e_1=e_{i+j},
\end{align*}
we derive that
\[e_ie_j=e_{i+j}, \quad 2\leq i+j\leq n-p.\]
\end{proof}

\begin{pr}\label{fili}
  An $n$-dimensional associative algebra $\mathcal{A}$ is filiform if and only if the nilindex of $\mathcal{A}$ is equal $n$.
\end{pr}
\begin{proof} If $\mathcal{A}$ is filiform then by definition implies that the nilindex of $\mathcal{A}$ is equal $n$. Let us suppose that the nilindex of $\mathcal{A}$ is equal $n$.
 It follows $\mathcal{A}^n=0$ and $\mathcal{A}^{n-1}\neq0$. Then we obtain a decreasing chain of ideals $\mathcal{A}\supset\mathcal{A}^2\supset\mathcal{A}^3\supset\dots\supset\mathcal{A}^{n-1}\supset\mathcal{A}^n=\{0\}$.
  Obviously, there are two possibilities $\dim\mathcal{A}^2=n-1$ or $\dim\mathcal{A}^2=n-2$. Assume that $\dim\mathcal{A}^2=n-1$.
   Choose a basis $\{e_1,e_2,\dots,e_n\}$ such that
    $e_1\in \mathcal{A}\setminus\mathcal{A}^2, \dots, \{e_i,e_{i+1}\}\subset \mathcal{A}^i\setminus\mathcal{A}^{i+1}, e_{i+2}\in \mathcal{A}^{i+1}\setminus\mathcal{A}^{i+2}, \dots, e_{n}\in \mathcal{A}^{n-1}$, where $i\neq1$. Such a basis can be chosen always. We may assume that
\[e_i=\underbrace{(((e_1e_1)e_1)\cdots e_1)}_i+(\ast), \quad e_{i+1}=\underbrace{(((e_1e_1)e_1)\cdots e_1)}_i+(\ast\ast)\]
where $(\ast),(\ast\ast)\in\mathcal{A}^{i+1}$. Then $e_i-e_{i+1}\in\mathcal{A}^{i+1}$. It follows we obtain a contradiction with the assumption that $\dim\mathcal{A}^i\setminus\mathcal{A}^{i+1}=2$. Therefore $\dim\mathcal{A}^i\setminus\mathcal{A}^{i+1}=1$ where $2\leq i\leq n-1$ and $\dim\mathcal{A}^2=n-2$. Thus, $\dim\mathcal{A}^i=n-i$ where $2\leq i\leq n$, i.e. $\mathcal{A}$ is a filiform associative algebra.
\end{proof}

The proof of the previous proposition implies the next proposition.
\begin{pr}\label{fili1}
Let $\mathcal{A}$ is $n$-dimensional associative algebra and $\dim\mathcal{A}^2=n-1$. Then $\mathcal{A}$ is a null-filiform associative algebra.
\end{pr}

Now we will classify the $n$-dimensional complex filiform associative algebras.

\begin{teo}\label{fili2}
For $n>3$, every $n$-dimensional complex filiform associative algebra is isomorphic to one of the next pairwise non-isomorphic algebras with basis $\{e_1,e_2,\dots,e_n\}$:
\[\begin{array}{llll}
\mu_{1,1}^n: & e_ie_j=e_{i+j},  & & \\
\mu_{1,2}^n: & e_ie_j=e_{i+j},  & e_ne_n=e_{n-1}, & \\
\mu_{1,3}^n: & e_ie_j=e_{i+j},  & e_1e_n=e_{n-1}, & \\
\mu_{1,4}^n: & e_ie_j=e_{i+j},  & e_1e_n=e_{n-1}, & e_ne_n=e_{n-1}
\end{array}\]
where $2\leq i+j\leq n-1$.
\end{teo}
\begin{proof} By Theorem~\ref{pfil} when $p=1$ we obtain a filiform associative algebra. So we have that
\[e_ie_j=e_{i+j}, \quad 2\leq i+j\leq n-1, \qquad e_1e_n=\alpha e_{n-1}, \qquad e_ne_n=\beta e_{n-1}.\]

  We make the following general transformation of basis:
\[e_1^\prime=\sum\limits_{k=1}^nA_ke_k, \quad e_n^\prime=\sum\limits_{k=1}^nB_ke_k, \quad A_1(A_1B_n-A_nB_1)\neq0\]
while the other elements of the new basis (i.e. $e_k^\prime,2\leq k\leq n-1$) are obtained as products of the above elements.

We can obtain that
\begin{align*}
e_2^\prime& =A_1^2e_2+\sum\limits_{k=3}^{n-2}\sum\limits_{s=1}^{k-1}A_sA_{k-s}e_k+
(\alpha A_1A_n+\beta A_n^2+\sum\limits_{s=1}^{n-2}A_sA_{n-s-1})e_{n-1},\\
e_k^\prime &=A_1^ke_k+kA_1^{k-1}A_2e_{k+1}+\sum\limits_{s=k+2}^{n-1}\theta_{k,s}(A_1,A_2,\dots,A_n)e_s, \quad 3\leq k\leq n-1.
\end{align*}

From the following multiplication  table in this new basis
\begin{align*}
0&=e_n^\prime e_{n-2}^\prime =(\sum\limits_{k=1}^nB_ke_k)(A_1^{n-2}e_{n-2}+(\ast))=A_1^{n-2}B_1e_{n-1}, \\
0 &=e_n^\prime e_{n-3}^\prime  =(\sum\limits_{k=2}^nB_ke_k)(A_1^{n-3}e_{n-3}+(\ast))=A_1^{n-3}B_2e_{n-1},\\
0& =e_n^\prime e_{n-4}^\prime=(\sum\limits_{k=3}^nB_ke_k)(A_1^{n-4}e_{n-4}+(\ast))=A_1^{n-4}B_3e_{n-1},\\
&  \qquad \qquad \qquad  \qquad \qquad \vdots \\
0&=e_n^\prime e_{n-k}^\prime=(\sum\limits_{s=k-1}^nB_se_s)(A_1^{n-k}e_{n-k}+(\ast))=A_1^{n-k}B_{k-1}e_{n-1},\\
&  \qquad \qquad \qquad  \qquad \qquad  \vdots \\
0&=e_n^\prime e_2^\prime=(B_{n-3}e_{n-3}+B_{n-2}e_{n-2}+B_{n-1}e_{n-1}+B_ne_n)(A_1^2e_2+(\ast))=A_1^2B_{n-3}e_{n-1}, \\
0&=e_n^\prime e_1^\prime= (B_{n-2}e_{n-2}+B_{n-1}e_{n-1}+B_ne_n)(\sum\limits_{s=1}^nA_se_s)=(A_1B_{n-2}+\beta A_nB_n)e_{n-1},
\end{align*}
we have the following restrictions on the coefficients:
\[B_k=0, \quad 1\leq k\leq n-3, \quad B_{n-2}=-\frac{\beta A_nB_n}{A_1}, \quad A_1B_n\neq0.\]
Calculating new parameters we obtain:
\[\alpha^\prime=\frac{\alpha B_n}{A_1^{n-2}}, \qquad \beta^\prime=\frac{\beta B_n^2}{A_1^{n-1}}.\]

If $\alpha=0$, then we have $\alpha^\prime=0$.

\begin{itemize}
  \item If $\beta=0$, it follows $\beta^\prime=0$. Obviously we obtain $\mu_{1,1}^n$.
  \item If $\beta\neq0$, then choose $B_n=A_1^2/\beta^{1/2}$ obtain $\beta^\prime=1$. Hence we have $\mu_{1,2}^n$.
\end{itemize}

If $\alpha\neq0$, then by choosing $B_n=A_1^3/\alpha$ we have $\alpha^\prime=1$ and $\beta^\prime=\beta A_1^3/\alpha^2$.

\begin{itemize}
  \item If $\beta=0$, it follows $\beta^\prime=0$. So we have $\mu_{1,3}^n$.
  \item If $\beta\neq0$, then choose $A_1=\sqrt[3]{\alpha^2/\beta}$ and implies that $\beta^\prime=1$. Obviously we have $\mu_{1,4}^n$.
\end{itemize}

\end{proof}

\section{Naturally graded quasi-filiform associative algebra}

Let $\mathcal{A}$ be a naturally graded $n$-dimensional quasi-filiform associative algebra. Then, there are two possibilities for the characteristic sequence, either  $C(\mathcal{A})=(n-2,1,1)$ or $C(\mathcal{A})=(n-2,2)$. So we start to considering the case $C(\mathcal{A})=(n-2,1,1)$.

\begin{de}\label{2-fil}
  An associative algebra $\mathcal{A}$ is called $2$-filiform if $C(\mathcal{A})=(n-2,1,1)$.
\end{de}

By definition of characteristic sequence the operator $L_{e_1}$ in the Jordan form has one block $J_{n-2}$ of size $n-2$ and $2$ blocks $J_1$ (where $J_1=\{0\}$)  of size one.

The possible forms for the operator $L_{e_1}$ are the following:
\[
\begin{pmatrix}
J_{1}& 0&0 \\
0& J_{n-2}& 0  \\
0&0&J_{1}
\end{pmatrix},
\qquad
\begin{pmatrix}
J_{n-2}& 0&0 \\
0& J_{1} & 0  \\
0&0&J_{1}
\end{pmatrix}.
\]

Let us suppose that the operator $L_{e_1}$ has the first form. Then we have the next multiplications
\[
\left\{
\begin{aligned}
  e_1e_1&=0, \\
  e_1e_i&=e_{i+1}, \quad 2\leq i\leq n-2,\\
  e_1e_n&=0. \\
  \end{aligned}
\right.
\]

From the chain of equalities
\[e_4=e_1e_3=e_1(e_1e_2)=(e_1e_1)e_2=0\]
we obtain a contradiction.

Thus, we can reduce the study to the following form of the matrix  $L_{e_1}$:
\[\begin{pmatrix}
J_{n-2}& 0&0 \\
0& J_{1} & 0  \\
0&0&J_{1}
\end{pmatrix}.
\]

So, there exists a basis $\{e_1,e_2,\dots,e_{n-2},f_1,f_2\}$ such that
\[
\begin{array}{ll}
  e_1e_i=e_{i+1}, &  \quad 1\leq i\leq n-3, \\
  e_1f_1=0, \\
  e_1f_2=0.
\end{array}
\]

Applying arguments similar to the Case 1 of Theorem~\ref{natpfil}, we obtain the products:
\begin{align*}
  e_ie_j&=e_{i+j}, \qquad 2\leq i+j\leq n-2, \\
  e_kf_s&=0, \qquad \quad \ 1\leq k\leq n-2, \quad 1\leq s\leq 2.
\end{align*}

By multiplication we have
\[e_1\in\mathcal{A}_1, \quad e_2\in\mathcal{A}_2, \quad \dots, \quad e_{n-2}\in\mathcal{A}_{n-2},\]
but we do not know about the places of the basis $\{f_1,f_2\}$.

Let denote by $t_1,t_2$ the places of basis elements $\{f_1,f_2\}$ in the  natural graduation corresponding, that $f_i\in\mathcal{A}_{t_i}$, where $1\leq i\leq 2$.
 Further the law of the algebra with set $\{t_1,t_2\}$ will denote by $\mu(t_1,t_2)$. We can suppose that $1\leq t_1\leq t_2$.

\begin{pr}
Let $\mathcal{A}$ be a naturally graded $2$-filiform associative algebra. Then $t_i\leq i$ for any $1\leq i\leq 2$.
\end{pr}
\begin{proof}
In fact  $t_1=1$, since if $t_1>1$, then the algebra $\mathcal{A}$ is one generated and by Proposition~\ref{fili1}, it is a null-filiform associative algebra that is a contradiction.

We will prove by contradiction, that is suppose that $t_2>2$. Then
\[\mathcal{A}_1=\{e_1,f_1\}, \quad \mathcal{A}_2=\{e_2\}, \quad \dots, \quad \mathcal{A}_{t_2}=\{e_{t_2},f_2\}, \quad \dots \]
It follows
\[e_ie_j=e_{i+j}, \quad  2\leq i+j\leq n-p, \quad e_1f_1=0, \quad f_1e_1=\alpha e_2. \]

From the chain of equalities
\[0=(e_1f_1)e_1=e_1(f_1e_1)=\alpha e_3,\]
we obtain that $f_1e_1=0$.

By considering the next identities
\[f_1e_{t_2-1}=f_1(e_1e_{t_2-2})=(e_1f_1)e_{t_2-2}=0, \quad e_{t_2-1}f_1=(e_{t_2-2}e_1)f_1=e_{t_2-2}(e_1f_1)=0,\]
we obtain that $f_2\notin \mathcal{A}_{t_2}$, it is contradiction.
\end{proof}

Note that the case of $\mu(1,1)$ is a filiform algebra of degree $2$. Thus, it is sufficient to consider the case $\mu(1,2)$.

\begin{pr} Let $\mathcal{A}$ be a $5$-dimensional complex naturally graded $2$-filiform non-split associative algebras of type $\mu(1,2)$. Then $\mathcal{A}$ is isomorphic to one of the following pairwise non-isomorphic algebras:
\[\lambda_1:\left\{\begin{array}{l}
e_1e_1=e_2,\\
e_1e_2=e_2e_1=e_3,\\
e_4e_1=e_5\end{array}\right.
\qquad
\lambda_2:\left\{\begin{array}{l}
e_1e_1=e_2,\\
e_1e_2=e_2e_1=e_3,\\
e_4e_1=e_5,\\
e_4e_2=e_5e_1=e_3
\end{array}\right.\]
\end{pr}

\begin{proof} We may assume
\[\mathcal{A}_1=\{e_1,e_4\}, \quad \mathcal{A}_2=\{e_2,e_5\}, \quad \mathcal{A}_3=\{e_3\}.\]

Now we have the next multiplications
\[\begin{array}{ll}
    e_ie_j=e_{i+j}, & \qquad 2\leq i+j\leq 3,  \\
    e_4e_1=\alpha_1e_2+\gamma_1e_5, & \qquad e_4e_2=\alpha_2e_3, \\
    e_4e_4=\alpha_4e_2+\gamma_2e_5, &  \qquad e_4e_5=\alpha_5e_3, \\
    e_5e_1=\beta_1e_3, &  \qquad e_5e_4=\beta_4e_3. \\

  \end{array}
\]

Using the identities we obtain the next restrictions:
\[\left\{\begin{aligned}
\alpha_1&=\alpha_4=0, \\
\beta_4\gamma_1 &=0,\\
\alpha_2&=\beta_1\gamma_1,\\
(\alpha_5-\beta_4)\gamma_2&=0, \\
\alpha_5\gamma_1-\beta_1\gamma_2&=0.
\end{aligned}\right.
\]

According to the characteristic sequence, we derive that $\rank(L_{e_1+Ae_4})\leq 2(A\neq0)$, this implies that $\gamma_2=0$ and $\gamma_1\neq0$. A simple transformation of the basis implies that $\gamma_1=1$.
 We deduce that $\alpha_5=\beta_4=0$ and $\beta_1=\alpha_2$. Thus the products are the following:
\[\left\{\begin{aligned}
  e_ie_j&=e_{i+j}, & \quad 2\leq i+j\leq 3, \\
  e_{4}e_1&=e_5, &  \\
  e_{4}e_2&=\alpha_2e_3, &  \\
  e_{5}e_1&=\alpha_2e_3. &
  \end{aligned}\right.
\]

Make the change of generator basis elements
\[e_1^\prime=A_1e_1+A_2e_4, \quad e_4^\prime=B_1e_1+B_2e_4, \quad A_1B_2-A_2B_1\neq0.\]

The equality $e_1^\prime e_4^\prime=0$ deduces $B_1=0$. Calculating new parameters we obtain:
\[\alpha_2^\prime=\frac{\alpha_2B_4}{A_1+\alpha_2A_4}.\]

If $\alpha_2=0$, then $\alpha_2^\prime=0$. So we obtain the algebra $\lambda_1$.

If $\alpha_2\neq0$, then putting $B_4=\frac{A_1+\alpha_2A_4}{\alpha_2}$, we obtain $\alpha_2^\prime=1$. Thus, in this case we get the algebra $\lambda_2$.
\end{proof}

\begin{teo} Let $\mathcal{A}$ be an $n$-dimensional ($n>5$) complex naturally graded 2-filiform non-split associative algebras of type $\mu(1,2)$. Then $\mathcal{A}$ is isomorphic to the following algebra:
\[\mu_{2,1}^n:\left\{\begin{aligned}
  e_ie_j&=e_{i+j}, & \quad 2\leq i+j\leq n-2, \\
  e_{n-1}e_1&=e_n. &  \\
  \end{aligned}\right.
\]
\end{teo}

\begin{proof}
We denote by $e_{n-1},e_n$ the elements $f_1, f_2$.

According to the theorem conditions we have the following multiplications
\[\left\{\begin{aligned}
  e_ie_j &=e_{i+j}, && 2\leq i+j\leq n-2, \\
  e_{n-1}e_1 &=\alpha_1e_2+\gamma_1e_n, &  \\
  e_{n-1}e_i &=\alpha_ie_{i+1},&& 2\leq i\leq n-3, \\
  e_{n-1}e_{n-1} & =\alpha_{n-1}e_2+\gamma_2e_n, &  \\
  e_{n-1}e_n & =\alpha_ne_3, &  \\
  e_ne_i & =\beta_ie_{i+2}, && 1\leq i\leq n-4, \\
  e_ne_{n-1} & =\beta_{n-1}e_3, &  \\
  e_ne_n & =\beta_ne_4, &
\end{aligned}\right.
\]
where $(\gamma_1,\gamma_2)\neq(0,0)$.

The equalities
\[0=(e_1e_n)e_i=e_1(e_ne_i)=\beta_ie_1e_{i+2}=\beta_ie_{i+3}\]
yield $\beta_i=0, \ 1\leq i\leq n-5$.

Moreover, from
\[\beta_{n-4}e_{n-2}=e_ne_{n-4}=e_n(e_1e_{n-5})=(e_ne_1)e_{n-5}=0\]
we obtain that $\beta_{n-4}=0$.

Considering the identity
\[0=(e_1e_n)e_{n-1}=e_1(e_ne_{n-1})=\beta_{n-1}e_1e_3=\beta_{n-1}e_4, \]
we can deduce that $\beta_{n-1}=0$.

From identities $e_6(e_5e_1)=(e_6e_5)e_1, \ e_6(e_5e_5)=e_6(e_5e_5)$, with $n=6$, and
from $e_1(e_ne_{n})=(e_1e_n)e_{n}$,
with $n>6$, we obtain that $\beta_n=0$.

Analogously, from identities
\[(e_1e_{n-1})e_i=e_1(e_{n-1}e_i), \quad (e_{n-1}e_1)e_{n-4}=e_{n-1}(e_1e_{n-4})\]
we derive that
\[\alpha_{i}=\alpha_{n-1}=\alpha_n=0, \quad 1\leq i\leq n-3.\]
Thus, we obtain the following multiplication table
\[\left\{\begin{aligned}
  e_ie_j&=e_{i+j}, &  \quad 2\leq i+j\leq n-2, \\
  e_{n-1}e_1&=\gamma_1e_n, &  \\
  e_{n-1}e_{n-1}&=\gamma_2e_n &
 \end{aligned}\right.
\]
where $(\gamma_1,\gamma_2)\neq(0,0)$.

Consider for the element $e_1 + Ae_{n-1} (A\neq0)$ the operator $L_{e_1+Ae_{n-1}}$. Then we have
\[L_{e_1+Ae_{n-1}}=
    \begin{pmatrix}
      0 & I_{n-3} &  & 0 & A\gamma_1 \\
      \vdots & \vdots &  & \vdots & \vdots \\
      0 & 0 & \cdots & 0 & 0 \\
      0 & 0 & \cdots & 0 & A\gamma_2 \\
      0 & 0 & \cdots & 0 & 0
    \end{pmatrix},
\]
where $I_{n-3}$ is the unit matrix of size $n-3$.


Since $\rank(L_{e_1+Ae_{n-1}})\leq n-3$ (otherwise, the characteristic sequence for the
element $e_1+Ae_{n-1}$ would be bigger than $(n-2,1,1)$), we conclude that $A\gamma_{2}=0$, hence $\gamma_{2}=0$ and $\gamma_1\neq0$. By recalling of the basis, one may assume that $\gamma_1=1$.
\end{proof}

Now we will consider the case of the characteristic sequence equal to $(n-2,2)$

According to the definition of characteristic sequence, $C(\mathcal{A})=(n-2,2)$, it follows the existence of a basis element $e_1\in\mathcal{A}\setminus\mathcal{A}^2$ and a basis $\{e_1,e_2,\dots,e_n\}$ such that the left multiplication  operator $L_{e_1}$ has one of the following forms:
\[
\begin{pmatrix}
J_{2}& 0 \\
0& J_{n-2}
\end{pmatrix},
\qquad
\begin{pmatrix}
J_{n-2}& 0 \\
0& J_2
\end{pmatrix}.
\]

Let us suppose that the operator $L_{e_1}$ has the first form. Then we have the next multiplications
\[\left\{\begin{aligned}
e_1e_1&=e_2,& \\
e_1e_2&=0,&\\
e_1e_i&=e_{i+1},& \quad 3\leq i\leq n-1\\
e_1e_n&=0.& \\
\end{aligned}\right.
\]

From the chain of equalities
\begin{align*}
0& =(e_1e_2)e_3=(e_1(e_1e_1))e_3=((e_1e_1)e_1)e_3=(e_2e_1)e_3=e_2(e_1e_3)\\
&=e_2e_4=(e_1e_1)e_4=e_1(e_1e_4)=e_1e_5=e_6,
\end{align*}
we obtain contradiction.

Thus, we can reduce the study to the following form of the matrix  $L_{e_1}$:
\[\begin{pmatrix}
J_{n-2}& 0 \\
0& J_2
\end{pmatrix}.\]

\begin{teo} Let $\mathcal{A}$ be $5$-dimensional complex naturally graded associative algebra with characteristic sequence $C(\mathcal{A})=(3,2)$. Then it is isomorphic to one of the following pairwise non-isomorphic algebras:

\[\pi_1(\alpha): \left\{\begin{array}{l}
e_1e_1=e_2, \\
e_1e_2=e_2e_1=e_3,\\
e_1e_4=e_5,\\
e_{4}e_1=\alpha e_5, \alpha\in\mathds{C}
\end{array}\right.
\quad
\pi_2: \left\{\begin{array}{l}
e_1e_1=e_2, \\
e_1e_2=e_2e_1=e_3,\\
e_1e_4=e_4e_1=e_5,\\
e_{4}e_4=e_5
\end{array}\right.
\quad
\pi_3: \left\{\begin{array}{l}
e_1e_1=e_2, \\
e_1e_2=e_2e_1=e_3,\\
e_1e_4=e_5,\\
e_4e_4=e_5
\end{array}\right.\]

\[\pi_4:\left\{\begin{array}{l}
e_1e_1=e_2, \\
e_1e_2=e_2e_1=e_3,\\
e_1e_4=e_5,\\
e_4e_1=e_2-e_5,\\
e_5e_1=e_3
\end{array}\right.
\quad
\pi_5:\left\{\begin{array}{l}
e_1e_1=e_2,\\
e_1e_2=e_2e_1=e_3,\\
e_1e_4=e_5,\\
e_4e_1=e_2+e_5,\\
e_4e_2=2e_3,\\
e_4e_4=2e_5, \\
e_5e_1=e_3
\end{array}\right.
\quad
\pi_6:\left\{\begin{array}{l}
e_1e_1=e_2,\\
e_1e_2=e_2e_1=e_3,\\
e_1e_4=e_4e_1=e_5,\\
e_4e_4=e_2, \\
e_4e_5=e_5e_4=e_3
\end{array}\right.\]

\[\pi_7:\left\{\begin{array}{l}
e_1e_1=e_2, \\
e_1e_2=e_2e_1=e_3,\\
 e_1e_4=-e_4e_1=e_5,\\
e_4e_4=e_2,\\
e_5e_4=-e_4e_5=e_3
\end{array}\right.
\quad
\pi_8(\alpha):\left\{\begin{array}{ll}
e_1e_1=e_2, & e_1e_2=e_2e_1=e_3,\\
e_1e_4=e_5, &  e_4e_1=(1-\alpha)e_2+\alpha e_5,  \\
e_4e_2=(1-\alpha^2)e_3, &  e_4e_4=-\alpha e_2+(1+\alpha)e_5,   \\
e_4e_5=-\alpha^2e_3, & e_5e_1=(1-\alpha)e_3, \\
e_5e_4=-\alpha e_3, & \alpha\in\mathds{C}.
\end{array}\right.\]
\end{teo}
\begin{proof}
We can assume that
\[\mathcal{A}_1=\{e_1,e_4\}, \quad \mathcal{A}_2=\{e_2,e_5\}, \quad \mathcal{A}_3=\{e_3\}.\]

By the identities associativity and from the  left multiplication operator $L_{e_1}$ we deduce
  \[\begin{array}{lll}
e_1e_1=e_2,& e_4e_1=\alpha_1e_2+\alpha_2e_5,& e_4e_5=\alpha_2\beta_1e_3,  \\
e_1e_2=e_2e_1=e_3,&e_4e_2=\alpha_1(\alpha_2+1)e_3,&e_5e_1=\alpha_1e_3,\\
e_1e_4=e_5,& e_4e_4=\beta_1e_2+\beta_2e_5,&e_5e_4=\beta_1e_3,\\
\end{array}\]
with the next restrictions
\begin{equation}\label{a1}\alpha_1\beta_2=\alpha_1^2(\alpha_2+1)+\beta_1(\alpha_2^2-1), \quad \beta_1(\alpha_1(\alpha_2+1)+\beta_2(\alpha_2-1))=0.\end{equation}

Let us suppose $\boxed{e_5\in Z(\mathcal{A})}$, it follows $\alpha_1=\beta_1=0$. Then we have
\[e_1e_1=e_2, \quad e_1e_2=e_2e_1=e_3, \quad e_1e_4=e_5, \quad e_4e_1=\alpha_2e_5, \quad e_4e_4=\beta_2e_5.\]

Make the change of generator basis elements
\[e_1^\prime=A_1e_1+A_2e_4, \quad e_4^\prime=B_1e_1+B_2e_4, \quad A_1(A_1+\beta_2A_2)(A_1B_2-A_2B_1)\neq0.\]

The equality $e_1^\prime e_5^\prime=0$, implies $B_1=0$. Calculating new parameters we obtain:
\[\alpha_2^\prime=\frac{\alpha_2A_1+\beta_2A_2}{A_1+\beta_2A_2}, \quad \beta_2^\prime=\frac{\beta_2B_2}{A_1+\beta_2A_2}.\]

If $\beta_2=0$, then $\beta_2^\prime=0$ and $\alpha_2^\prime=\alpha_2$ and we obtain $\pi_1(\alpha)$.

If $\beta_2\neq0$, then by choosing $B_2=\frac{A_1+\beta_2A_2}{\beta_2}$  we deduce $\beta_2^\prime=1$. So we have the next invariant expression
\[\alpha_2^\prime-1=\frac{(\alpha_2-1)A_1}{A_1+\beta_2A_2}.\]

\begin{itemize}
  \item If $\alpha_2=1$, then $\alpha_2^\prime=1$ and we have $\pi_2$.
  \item If $\alpha_2\neq1$, then by choosing $A_2=-\frac{\alpha_2A_1}{\beta_2}$ we get $\alpha_2^\prime=0$ and obtain $\pi_3$.
\end{itemize}

Let us suppose $\boxed{e_5\notin Z(\mathcal{A})}$ following $(\alpha_1,\beta_1)\neq(0,0)$. Then we consider the next cases.

\textbf{Case 1.} Let $\boxed{e_5\in Z^r(\mathcal{A})}$. Then $\alpha_2\beta_1=0$.

\textbf{Case 1.1.} If $\beta_1=0$, then $\alpha_1\neq0$ and $\beta_2=\alpha_1(\alpha_2+1)$.
Analogously as the previous case we make change of generator basis
\[e_1^\prime=A_1e_1+A_2e_4, \qquad e_4^\prime=B_1e_1+B_2e_4,\]
\[ A_1(A_1+\alpha_1A_2)(A_1+\alpha_1A_2+\alpha_1\alpha_2A_2)(A_1B_2-A_2B_1)\neq0,\]
and find restrictions on the coefficients: $B_1=0$.

Calculating new parameters we obtain
\[\alpha_1^\prime=\frac{\alpha_1B_2}{A_1+\alpha_1A_2}, \quad \alpha_2^\prime=\frac{\alpha_2A_1}{A_1+\alpha_1A_2+\alpha_1\alpha_2A_2}, \quad \alpha_2^\prime+1=\frac{(\alpha_2+1)(A_1+\alpha_1A_2)}{A_1+\alpha_1A_2+\alpha_1\alpha_2A_2}.\]

By choosing $B_2=\frac{A_1+\alpha_1A_2}{\alpha_1}$ we obtain that $\alpha_1^\prime=1$.

If $\alpha_2=0$ then $\alpha_2^\prime=0$, and we have $\pi_8(0)$.

If $\alpha_2\neq0$ then we consider the next cases:
\begin{itemize}
  \item If $\alpha_2=-1$ then $\alpha_2^\prime=-1$ and we obtain $\pi_4$.
  \item If $\alpha_2\neq-1$ then by choosing $A_2=\frac{(\alpha_2-1)A_1}{\alpha_1(\alpha_2+1)}$, we deduce $\alpha_2^\prime=1$ and we have $\pi_5$.
\end{itemize}

\textbf{Case 1.2.} If $\beta_1\neq0$ then $\alpha_2=0$. By restrictions \eqref{a1} we obtain that $\beta_2=\alpha_1$ and $\beta_1=0$ that is contradiction of condition.

\textbf{Case 2.} Let $\boxed{e_5\notin Z^r(\mathcal{A})}$. Then $\alpha_2\beta_1\neq0$. Taking the next change of basis \[e_i^\prime=e_i, \quad 1\leq i\leq 3, \quad e_4^\prime=\frac{1}{\sqrt{\beta_1}}e_4, \ e_5^\prime=\frac{1}{\sqrt{\beta_1}}e_5\] we can assume that $\beta_1=1$.
%

\textbf{Case 2.1.} Let $\alpha_1=0$. Then from restriction \eqref{a1} we have
$\left\{\begin{aligned}
           \alpha_2^2-1&=0, \\
           (\alpha_2-1)\beta_2&=0.
         \end{aligned}\right.$

\textbf{Case 2.1.1.} Let $\alpha_2=1$. Then we have
\begin{align*}
e_1e_1&=e_2, \qquad \qquad  e_1e_2=e_2e_1=e_3, \quad e_1e_4=e_4e_1=e_5, \\
e_4e_4&=e_2+\beta_2e_5,\quad e_4e_5=e_5e_4=e_3.
\end{align*}

By making the change of generator basis elements, calculating restrictions on the coefficients and new parameters and applying similar arguments as the previous cases we obtain $\pi_8(1)$ and $\pi_6$.

\textbf{Case 2.1.2.} If $\alpha_2\neq1$ then $\alpha_2=-1$ and $\beta_2=0$. Hence we obtain $\pi_7$.

Note that the algebra $\pi_6$ is commutative and algebra $\pi_7$ is noncommutative. Thus, algebras $\pi_6$ and $\pi_7$ are non-isomorphic.

\textbf{Case 2.2.} Let $\alpha_1\neq0$. Then we consider the next cases.

\textbf{Case 2.2.1.} If $\alpha_2=-1$ then $\beta_2=0$. Make the change of generator basis elements
\[e_1^\prime=A_1e_1+A_2e_4, \quad e_4^\prime=B_1e_1+B_2e_4, \quad A_1(A_1^2+\alpha_1A_1A_2+A_2^2)(A_1B_2-A_2B_1)\neq0.\]
Calculating restrictions on the coefficients and new parameters we obtain:
\[B_1=0, \quad B_2=(A_1^2+\alpha_1A_1A_2+A_2^2)^{1/2}, \quad \alpha_1^\prime=\frac{\alpha_1A_1+2A_2}{(A_1^2+\alpha_1A_1A_2+A_2^2)^{1/2}}.\]

It is easy to check that the next expression is an invariant expression
\[(\alpha_1^\prime)^2-4=\frac{(\alpha_1^2-4)A_1^2}{A_1^2+\alpha_1A_1A_2+A_2^2}.\]

If $\alpha_1^2=4$ then $\alpha_1^\prime=\pm2$ and we have $\pi_8(-1)$.

If $\alpha_1^2\neq4$ then by choosing $A_2=-\frac{1}{2}\alpha_1A_1$ we obtain $\alpha_1^\prime=0$ and it follows we have $\pi_7$.

\textbf{Case 2.2.2.} Let $\alpha_2\neq-1$. Then from restriction \eqref{a1} we get $\alpha_2\neq1, \beta_2\neq0$ and
\[\alpha_1=\frac{(1-\alpha_2)\beta_2}{1+\alpha_2}, \quad (\alpha_2+1)^2+\alpha_2\beta_2^2=0.\]

By taking the next change of basis
\[e_i^\prime=e_i, \quad 1\leq i\leq3, \quad e_4^\prime=\frac{1+\alpha_2}{\beta_2}e_4, \quad e_5^\prime=\frac{1+\alpha_2}{\beta_2}e_5,\]
we obtain $\pi_8(\alpha)$ where $\alpha\notin\{\pm1,0\}$.

By making the general change of basis, using  multiplication table in a new basis and
calculating new parameters we obtain that $\alpha^\prime=\alpha$. It means that in the different values of $\alpha$, the obtained algebras are non-isomorphic.
\end{proof}

\begin{teo} Let $\mathcal{A}$ be $n$-dimensional $(n>5)$ complex naturally graded associative algebra with characteristic sequence $C(\mathcal{A})=(n-2,2)$. Then it is isomorphic to one of the following pairwise non-isomorphic algebras:
\[\mu_{2,2}^n(\alpha):\left\{\begin{array}{l}
e_ie_j=e_{i+j}, \\
e_1e_{n-1}=e_n,\\
e_{n-1}e_1=\alpha e_n, \\
\end{array}\right. \quad
\mu_{2,3}^n: \left\{\begin{array}{ll}
e_ie_j=e_{i+j}, \\
e_1e_{n-1}=e_n,\\
e_{n-1}e_1=e_n,\\
e_{n-1}e_{n-1}=e_n, \\
\end{array}\right.
\quad
 \mu_{2,4}^n:\left\{\begin{array}{l}
e_ie_j=e_{i+j}, \\
e_1e_{n-1}=e_n,\\
e_{n-1}e_{n-1}=e_n,\\
\end{array}\right.
\quad
\]
where $\alpha\in\mathds{C}$ and  \, $2\leq i+j\leq n-2$.
\end{teo}

\begin{proof} According to the condition of the theorem, we have the following multiplication:
\[\left\{\begin{aligned}
e_1e_i&=e_{i+1},& \quad 1\leq i\leq n-3, \\
e_1e_{n-2}&=0,&\\
e_1e_{n-1}&=e_n,&\\
e_1e_n&=0.& \\
\end{aligned}\right.
\]

It is easily seen that $\mathcal{A}_1=\{e_1,e_{n-1}\}, \mathcal{A}_2=\{e_2,e_n\}, \mathcal{A}_i=\{e_i\}$ for $3\leq i\leq n-2$. It follows $e_{n-2}\in Z(\mathcal{A})$.

Considering the next identities
\[e_1e_2=e_1(e_1e_1)=(e_1e_1)e_1=e_2e_1, \quad e_1e_3=e_1(e_2e_1)=(e_1e_2)e_1=e_3e_1,\]
we obtain that
\[e_1e_2=e_2e_1=e_3, \quad e_1e_3=e_3e_1=e_4.\]

Now we will prove the following equalities by an induction on $i$:
\begin{equation}\label{a2}e_1e_i=e_ie_1=e_{i+1}, \quad 1\leq i\leq n-3.
\end{equation}

Obviously, the equality holds for $i=2,3$. Let us assume that the equality \eqref{a2} holds for $2 < i < n - 3$, and we shall prove it for $i + 1$:
\[e_1e_{i+1}=e_1(e_ie_1)=(e_1e_i)e_1=e_{i+1}e_1,\]
so the induction proves the equalities \eqref{a2} for any $i, \ 2\leq i\leq n-3$.

From the following chain of equalities
\begin{align*}
e_ie_j&=e_i(e_1e_{j-1})=(e_ie_1)e_{j-1}=e_{i+1}e_{j-1}=\dots=e_{i+j-2}e_2\\
&=e_{i+j-2}(e_1e_1)=(e_{i+j-2}e_1)e_1=e_{i+j},
\end{align*}
we derive that
\[e_ie_j=e_{i+j}, \quad 2\leq i+j\leq n-2.\]

Let us introduce the notations:
\[e_{n-1}e_1=\gamma_1e_2+\alpha_1e_n, \quad e_{n-1}e_{n-1}=\gamma_2e_2+\alpha_2e_n, \quad e_{n-1}e_n=\beta_1e_3, \] \[e_ne_1=\beta_2e_3, \quad e_ne_{n-1}=\beta_3e_3, \quad e_ne_n=\beta_4e_4.\]

From the identities, we have
\[
\begin{array}{ccc}
\text{Identity }& & \text{ Constraint }\\[1mm]
\hline \hline\\[1mm]
(e_1e_n)e_1=e_1(e_ne_1)&\Longrightarrow &\beta_2=0, \\[1mm]
(e_1e_n)e_{n-1}=e_1(e_ne_{n-1})&\Longrightarrow &\beta_3=0,\\[1mm]
(e_{n-1}e_n)e_1=e_{n-1}(e_ne_1)&\Longrightarrow &\beta_1=0,\\[1mm]
(e_1e_{n-1})e_n=e_1(e_{n-1}e_n)&\Longrightarrow & \beta_4=0,\\[1mm]
(e_1e_{n-1})e_1=e_1(e_{n-1}e_1)&\Longrightarrow &\gamma_1=0, \\[1mm]
(e_1e_{n-1})e_{n-1}=e_1(e_{n-1}e_{n-1})&\Longrightarrow &\gamma_2=0.
\end{array}
\]

Considering the identities
\begin{align*}
e_ie_n&=(e_{i-1}e_1)e_n=e_{i-1}(e_1e_n)=0,\\
e_ne_i&=e_n(e_1e_{i-1})=(e_ne_1)e_{i-1}=0, && 2\leq i\leq n-4,\\
e_je_{n-1}&=(e_{j-1}e_1)e_{n-1}=e_{j-1}(e_1e_{n-1})=e_{j-1}e_n=0,\\
e_{n-1}e_j&=e_{n-1}(e_1e_{j-1})=(e_{n-1}e_1)e_{j-1}=\alpha_1e_ne_{j-1}=0, && 2\leq j\leq n-3,
\end{align*}
we deduce that
\[\left\{\begin{aligned}
e_ie_j&=e_{i+j},& \quad 2\leq i+j\leq n-2, \\
e_1e_{n-1}&=e_n,&\\
e_{n-1}e_1&=\alpha_1e_n,&\\
e_{n-1}e_{n-1}&=\alpha_2e_n.& \\
\end{aligned}\right.
\]

  We make the following general transformation of generators of the basis:
\[e_1^\prime=\sum\limits_{k=1}^nA_ke_k, \quad e_{n-1}^\prime=\sum\limits_{k=1}^nB_ke_k, \quad A_1(A_1+\alpha_2A_{n-1})(A_1B_{n-1}-A_{n-1}B_1)\neq0,\]
while the other elements of the new basis are obtained as products of the above elements.

From the next table of multiplications in this new basis
\[e_1^\prime e_n^\prime=0, \quad e_{n-1}^\prime e_1^\prime=\alpha_1^\prime e_n^\prime, \quad e_{n-1}^\prime e_{n-1}^\prime=\alpha_2^\prime e_n^\prime,\]
we have the following restrictions on the coefficients:
\[B_i=0, \quad 1\leq i\leq n-3,\]
and calculating new parameters we obtain:
\[\alpha_1^\prime=\frac{\alpha_1A_1+\alpha_2A_{n-1}}{A_1+\alpha_2A_{n-1}}, \qquad \alpha_2^\prime=\frac{\alpha_2B_{n-1}}{A_1+\alpha_2A_{n-1}}.\]

If $\alpha_2=0$, then $\alpha_2^\prime=0$ and $\alpha_1^\prime=\alpha_1$, and so   we have $\mu_{2,2}^n(\alpha)$.

If $\alpha_2\neq0$, then by choosing $B_{n-1}=\frac{A_1+\alpha_2A_{n-1}}{\alpha_2}$, we obtain $\alpha_2^\prime=1$.
It is easy to see that the next expression
\[\alpha_1^\prime-1=\frac{(\alpha_1-1)A_1}{A_1+\alpha_2A_{n-1}}\]
is an invariant expression.
\begin{itemize}
  \item If $\alpha_1=1$, then $\alpha_1^\prime=1$. Hence, we obtain $\mu_{2,3}^n$.
  \item If $\alpha_1\neq1$, then by choosing $A_{n-1}=-\frac{\alpha_1A_1}{\alpha_2}$, we obtain $\alpha_1^\prime=0$ and we have $\mu_{2,4}^n$.
\end{itemize}
\end{proof}

\end{document}